\documentclass[12pt]{article}
\usepackage{amsmath,mdframed}
\usepackage{amssymb}
\usepackage{latexsym}
\usepackage{amsthm}
\usepackage{mathrsfs}
\usepackage{enumitem}
\usepackage[colorlinks,
            linkcolor=red,
            anchorcolor=blue,
            citecolor=green
            ]{hyperref}
\usepackage{graphicx}
\usepackage{diagbox}
\usepackage{tikz}
\usepackage{tkz-graph}
\usetikzlibrary{graphs}
\usetikzlibrary{graphs.standard}
\usetikzlibrary{arrows,decorations.markings}
\tikzstyle{pointV}=[circle,fill=black,inner sep=0.5mm]

\newtheorem{thm}{Theorem}[section]
\newtheorem{lem}[thm]{Lemma}

\newtheorem{deff}[thm]{Definition}
\newtheorem{que}{Question}

\DeclareMathOperator{\red}{red}

\newcommand{\comm}[1]{#1}

\makeatother

\begin{document}

\begin{center}
{\large \bf  On shortening universal words for multi-dimensional permutations}
\end{center}

\begin{center}
Sergey Kitaev$^{1}$ and Dun Qiu$^{2}$\\[6pt]

$^{1}$ Department of Mathematics and Statistics \\
University of Strathclyde, 26 Richmond Street, Glasgow G1 1XH, UK\\[6pt]

$^{2}$Center for Combinatorics, LPMC, Nankai University, Tianjin 300071, P. R. China \\[6pt]

Email: $^{1}${\tt sergey.kitaev@strath.ac.uk},
           $^{2}${\tt qiudun@nankai.edu.cn}
\end{center}

\noindent\textbf{Abstract.} A universal word (u-word) for $d$-dimensional permutations of length $n$ is a 2-dimensional word with $d-1$ rows, any size $n$ window of which is order-isomorphic to exactly one permutation of length $n$, and all permutations of length $n$ are covered. It is known that u-words (in fact, even u-cycles, a stronger claim) for $d$-dimensional permutations exist. 

In this paper, we use the idea of incomparable elements to prove that u-words of length $(n!)^{d-1}+n-1-i(n-1)$, for $d\geq 2$ and  $$0\leq i\leq \frac{2^{d-1}}{n-1}\left[(1+(n-1)!)^{d-1}-\left(1+\frac{(n-1)!}{2}\right)^{d-1}\right],$$ for $d$-dimensional permutations of length $n$ exist, which generalizes the respective result of Kitaev, Potapov and Vajnovszki for ``usual'' permutations ($d=2$).



\section{Introduction}\label{intro}

There is a long line of research in the literature dedicated to the studies of universal objects of various types, in particular, the notion of a {\em universal cycle} ({\em u-cycle}) for combinatorial structures, generalizing the notion of a {\em de Bruijn sequence}, was introduced in \cite{CDG1992} and has been studied extensively for various objects. Considering the linear (i.e., non-cyclic) version of a universal cycle, we arrive to the notion of a {\em universal word} ({\em u-word}), also appearing in the literature in various places, e.g. in \cite{KPV}.

There are several ways to define the notion of a multi-dimensional permutation, and the visual motivation for the term ``multi-dimensional'' in one of the definitions is depicted, for example, in \cite[Fig. 2]{ZG07}. Following the definition in \cite{AKLPT}, the authors of this paper \cite{KQ}, using a greedy algorithm and arguing that the typical graph theoretical approach is not (easily) applicable, generalized the results in \cite{CDG1992,GKSZ,Hurlbert} by establishing the existence of u-cycles for $d$-permutations, from which it trivially follows that u-words for $d$-dimensional permutations exist.  The main goal of this paper is to generalize the results in \cite{KPV} by showing how incomparable elements can be used to shorten u-words for $d$-dimensional permutations. We will prove the following theorem that generalizes the respective result in \cite{KPV} (given by substituting $d=2$).

\begin{thm}\label{main-thm} U-words of length $(n!)^{d-1}+n-1-i(n-1)$, for $d\geq 1$ and 
$$0\leq i\leq \frac{2^{d-1}}{n-1}\left[(1+(n-1)!)^{d-1}-\left(1+\frac{(n-1)!}{2}\right)^{d-1}\right],$$
for $d$-dimensional permutations of length $n$ exist.\end{thm}

While the steps of the proof of Theorem~\ref{main-thm} are similar to those in \cite{KPV} to prove the respective result, there are several obstacles that we need to overcome to be able to produce results for $d$ dimensions, in particular, the notion of {\em twin permutations} needs to be defined properly. As it is mentioned in \cite{KQ}, extending results on universal objects for ``usual'' 2-dimensional permutations to $d$-dimensional analogues is not possible in a generic way such as taking products of universal cycles discussed in \cite{DiaconisGraham}.
 
Our work is related to a broader area of research on universal words (of combinatorial objects other than permutations) and shortenings of them, e.g.\ \cite{CKMS17,GGHKKKS}  for u-p-words/cycles, \cite{CK20} for shortening of u-cycles for word-patterns and set partitions, and \cite{CKKP24} for 2-dimensional u-p-words/cycles. 

The paper is organized as follows. In Section~\ref{prelim-sec} we give all necessary definitions and in Section~\ref{proof-sec} we prove Theorem~\ref{main-thm}. Concluding remarks are provided in Section~\ref{last-sec}. 

\section{Preliminaries}\label{prelim-sec}

\subsection{$d$-dimensional permutations} Our introduction of multi-dimensional permutations is largely inspired by the corresponding definitions in~\cite{CFKZ24,KQ}. Consider a permutation \( \pi = \pi_1 \pi_2 \dots \pi_n \) of length \( n \), that is, an element of the symmetric group \( S_n \). Throughout this paper, we use one-line notation for \( \pi \), but recall that its two-line form is given by  
\[
{\footnotesize \pi = \left(
\begin{array}{cccc}
1 & 2 & \dots & n \\
\pi_1 & \pi_2 & \dots & \pi_n
\end{array}
\right).}
\]

A \textit{\( d \)-dimensional permutation of length \( n \)}, or \textit{\( d \)-dimensional \( n \)-permutation}, is defined as an ordered \((d-1)\)-tuple \( (\pi^2, \pi^3, \dots, \pi^d) \) where each \( \pi^i = \pi^i_1 \pi^i_2 \dots \pi^i_n \in S_n \) for \( i = 2, \dots, d \). For instance, the triple \( (321, 123, 231) \) is a 4-dimensional permutation of length 3. We denote by \( S^d_n \) the set of all \( d \)-dimensional permutations of length \( n \). Note that \( S^2_n = S_n \), so ordinary permutations are naturally viewed as 2-dimensional.

We extend the two-line notation to a \textit{\( d \)-line notation}. Specifically, if we define \( \pi^1 = \pi^1_1 \pi^1_2\ldots\pi^1_n=  12\dots n \), then a \( d \)-dimensional permutation \( \Pi = (\pi^2, \pi^3, \dots, \pi^d) \) can be represented as a \( d \times n \) matrix:  
\renewcommand{\arraystretch}{1.25}
\[
\Pi = \left(
\begin{array}{cccc}
\pi^1_1 & \pi^1_2 & \dots & \pi^1_n \\
\pi^2_1 & \pi^2_2 & \dots & \pi^2_n \\
\pi^3_1 & \pi^3_2 & \dots & \pi^3_n \\
\vdots & & \dots & \vdots \\
\pi^d_1 & \pi^d_2 & \dots & \pi^d_n
\end{array}
\right).
\]
Alternatively, we may write this compactly as  
\[
\Pi = \left\{ \pi^i_j \right\}_{\substack{1 \le i \le d \\ 1 \le j \le n}}.
\]

By analogy with two-line notation, we interpret the \textit{columns} of the matrix as the \textit{elements} of \( \Pi \), denoted \( \Pi_1, \Pi_2, \dots, \Pi_n \), so that
\[
\Pi = \Pi_1 \Pi_2 \dots \Pi_n, \quad \text{where } \Pi_j = (\pi^1_j, \pi^2_j, \dots, \pi^d_j)^T = (j, \pi^2_j, \pi^3_j, \dots, \pi^d_j)^T.
\]

As a concrete example, let \( \Pi = (\pi^2, \pi^3) \in S^3_5 \), where \( \pi^2 = 21534 \) and \( \pi^3 = 41253 \). Then:
\renewcommand{\arraystretch}{1}
\[
{\footnotesize \Pi = \left(
\begin{array}{ccccc}
1 & 2 & 3 & 4 & 5 \\
2 & 1 & 5 & 3 & 4 \\
4 & 1 & 2 & 5 & 3
\end{array}
\right).}
\]
The corresponding elements of \( \Pi \) are then:  
\[
\Pi_1 = (1,2,4)^T, \quad \Pi_2 = (2,1,1)^T, \quad \Pi_3 = (3,5,2)^T, \quad \Pi_4 = (4,3,5)^T, \quad \Pi_5 = (5,4,3)^T.
\]

For a (2-dimensional) permutation, or word, $\pi$, the {\em reduced form} of $\pi$, denoted $\red(\pi)$, is obtained by replacing the $i$-th smallest element in $\pi$ by $i$. For example, $\red(5294)=3142$. Any $i$ consecutive elements of a word or permutation $w$ form a {\em factor} of $w$. If $u$ is a factor of $w$, we also say that $w$ {\em covers} $\red(u)$. Let $\pi$ be a 2-dimensional permutation of $\{1,\ldots,n\}$ and $x$ an element of $\{1,\ldots,n\}$. For $x<n$, we let $x^+$ denote a number $y$ such that $x<y<x+1$, while for $x=n$, $x^+=n+1$. Also, for $x>1$, we let $x^-$ denote an element $y$ such that $x-1<y<x$, while for $x=1$, $x^-=0$. The definitions of $x^+$ and $x^-$ can be generalized to any word instead of a permutation $\pi$ in a straightforward way, namely, $x^+$ refers to an element larger than $x$ but less than next largest element (if it exists), while $x^-$ refers to an element smaller than $x$ but larger than next smallest element (if it exists).

In what follows, we usually do not present the top row in multi-dimensional permutations, which is always the respective increasing permutation. Also, for a matrix $M$ in which each row has distinct elements (but there can be equal elements in different rows), the reduced form of $M$, denoted by $\red(M)$, is obtained by taking the reduced form of each row.  

\subsection{Universal words for $d$-dimensional permutations}

The following definition of a universal word is introduced in \cite{KQ}.

\begin{deff}\label{u-word-def} A matrix $U$ with $d-1$ rows is a {\em universal word}, or {\em u-word}, for $d$-dimensional $n$-permutations if {\em each} $d$-dimensional permutation (without the top row) can be found non-cyclically in $U$ {\em exactly once} as $n$ consecutive columns in the reduced form.
\end{deff}

For example, {\footnotesize $\left(\begin{array}{ccccc} 5 & 4 & 1 & 2 & 3 \\ 5 & 1 & 4 & 2 & 3 \end{array}\right)$} is a u-word for 3-dimensional 2-permutations. Indeed, the first two columns cover the permutation  {\footnotesize $\left(\begin{array}{cc} 1 & 2 \\ 2 & 1 \\ 2 & 1 \end{array}\right)$}, columns 2 and 3 cover the permutation  {\footnotesize $\left(\begin{array}{cc} 1 & 2 \\ 2 & 1 \\ 1 & 2 \end{array}\right)$}, and so on.

However, we need to extend the notion of a u-word introduced in Definition~\ref{u-word-def} to allow usage of incomparable elements. We begin with relevant ideas in \cite{KPV} for 2-dimensional permutations based on shortening of u-words via linear extensions of partially ordered sets (posets).

The word $11211$ is a u-word for all permutations of length 3 in the following sense, thus shortening a ``classical'' u-word for these permutations, say, $14524314$. Indeed, we treat equal elements as {\em incomparable elements}, while the relative order of these incomparable elements to the other elements must be respected. Thus, $112$ encodes all permutations whose last element is the largest one, namely, $123$ and $213$; starting at the second position, we obtain the word $121$ encoding the permutations $132$ and $231$, and finally, starting at the third position, we read the word $211$ encoding the permutations $312$ and $321$.   More generally, it is clear that the word 
$\underbrace{11\ldots 1}_{n-1\mbox{\tiny\ times}}2\underbrace{11\ldots 1}_{n-1\mbox{\tiny\ times}}=1^{n-1}21^{n-1}$ encodes all permutations and is of length $2n-1$ (instead of length $n!+n-1$ for earlier defined u-words for permutations). However, there are other compression possibilities creating u-words of lengths between $n$ and $n!+n-1$.  For example, the word $123212$ is also a u-word for permutations of length $3$. Note that the word of the form $11\ldots 1$ is the {\em trivial u-word} for all permutations of the respective length. 

In~\cite{KPV} it is shown that such u-words for $n$-permutations exist of lengths $n!+(1-i)(n-1)$ for $0\leq i\leq (n-2)!$. More specifically, the main concern in \cite{KPV}  is in the existence of u-words for permutations in which consecutive equal elements have exactly $n-1$ elements between them, and we have a similar concern in our paper.  However, to be able to prove Theorem~\ref{main-thm}, instead of introducing incomparable elements to be entire columns in a matrix in question, we refine the idea (and hence obtain many more possible lengths of u-words) by introducing incomparable elements in rows independently from each other. Hence, two columns, with $n-1$ columns between them, may have some rows with the same entry, and others with different entries.   For example, the matrix   {\footnotesize $\left(\begin{array}{ccc} 2 & 1 & 2 \\ 2 & 3 & 1 \\ 1 & 2 & 1 \end{array}\right)$} is interpreted by us, by taking independently linear extensions of the two pairs of incomparable elements, as an encoding of the four matrices, namely,  {\footnotesize $\left(\begin{array}{ccc} 2 & 1 & 3 \\ 2 & 3 & 1 \\ 1 & 3 & 2 \end{array}\right)$},  {\footnotesize $\left(\begin{array}{ccc} 3 & 1 & 2 \\ 2 & 3 & 1 \\ 1 & 3 & 2 \end{array}\right)$},  {\footnotesize $\left(\begin{array}{ccc} 2 & 1 & 3 \\ 2 & 3 & 1 \\ 2 & 3 & 1 \end{array}\right)$} and  {\footnotesize $\left(\begin{array}{ccc} 3 & 1 & 2 \\ 2 & 3 & 1 \\ 2 & 3 & 1 \end{array}\right)$}.

\subsection{Graph theory background}

A typical way to construct universal objects is via defining a suitable {\em transition graph}, in which the vertices correspond to the objects in question, and arguing that the graph is Hamiltonian. Showing Hamiltonicity often requires showing that another relevant graph is Eulerian. These notions are to be introduced next.

Let $G=(V,E)$ be a directed graph (digraph). For an edge $u\rightarrow v$ in $G$, $v$ is called the {\em head} and $u$ is called the {\em tail} of the edge. A {\em directed path} in $G$ is a sequence $v_1,\ldots,v_t$ of {\em distinct} nodes such that there is an edge $v_i\rightarrow v_{i+1}$ for each $1\leq i\leq t-1$. Such a path is  a {\em Hamiltonian path} if it contains all nodes in $G$. A closed Hamiltonian path ($v_t\rightarrow v_1$ is an edge) is a {\em Hamiltonian cycle}. If $G$ has a Hamiltonian cycle then $G$ is {\em Hamiltonian}. A digraph is {\em strongly connected} if there exists a directed path from any node to any other node.  A digraph is {\em connected} if for any pair of nodes $a$ and $b$ there exists a path in the underlying undirected graph (obtained from the digraph by removing all orientations). A {\em trail} in a digraph $G$ is a sequence $v_1,\ldots,v_t$ of nodes such that there is an edge $v_i\rightarrow v_{i+1}$ for each $1\leq i\leq t-1$ and edges are not visited more than once. An {\em Eulerian trail} in $G$ is a trail that goes through each edge exactly once. A closed Eulerian trail is an {\em Eulerian cycle}.  A digraph is {\em Eulerian}  if it has an Eulerian cycle. Let $d^+(v)$ (resp., $d^-(v)$) denote the out-degree (resp., in-degree)  of a node $v$, which is the number of edges pointing from (resp., to) $v$. A directed graph is {\em balanced} if $d^+(v)=d^-(v)$ for each node $v$ in the graph. The following result is well-known and is not hard to prove.

\begin{thm}\label{Eulerian-trail} A digraph $G$ is Eulerian if and only if it is balanced and (strongly) connected.\end{thm}

\subsection{Graphs of overlapping permutations and their clustering}

The cyclic version of a u-word, namely, when the factors of words can be read cyclically, is called a {\em universal cycle}, or {\em u-cycle}. The existence of u-cycles (of length $n!$) for $n$-permutations (from which the existence of u-words follows trivially by appending at the end the prefix of the word of length $n-1$)  was shown in  \cite{CDG1992} for any $n$ via {\em clustering} the {\em graph of overlapping $n$-permutations}. This graph has $n!$ vertices labelled by $n$-permutations, and there is an edge $x_1x_2\cdots x_n\rightarrow y_1y_2\cdots y_n$ if and only if the words $x_2x_3\cdots x_n$ and $y_1y_2\cdots y_{n-1}$ are order-isomorphic, that is, if and only if $x_i<x_j$ whenever $y_{i-1}<y_{j-1}$ for all $2\leq i<j\leq n$.

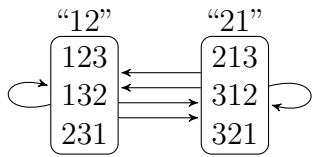
\begin{figure}[h]
\begin{center}
\comm{
\begin{tikzpicture}[->,>=stealth',shorten >=1pt,node distance=2cm,auto,main node/.style={rectangle,rounded corners,draw,align=center}]

\node[main node] (1) {123 \\ 132 \\  231}; 
\node (3) [above of=1,node distance=1cm] {``12''};
\node[main node] (2)  [right of=1] {213 \\ 312 \\  321}; 
\node (4) [above of=2,node distance=1cm] {``21''};

\path[draw,transform canvas={shift={(0,-0.1)}}] 
(1) edge node {} (2);
\path[draw,transform canvas={shift={(0,-0.3)}}] 
(1) edge node {} (2);

\path[draw,transform canvas={shift={(0,0.1)}}] 
(2) edge node {} (1);
\path[draw,transform canvas={shift={(0,0.3)}}] 
(2) edge node {} (1);

\path
(1) edge [loop left] node {} (1);
\path
(2) edge [loop right] node {} (2);

\end{tikzpicture}}
\vspace{-0.5cm}
\end{center}
\caption{Clustering the graph of overlapping 3-permutations}\label{clustering-order-3}
\end{figure}

A {\em pattern} of length $k$ is a permutation of $\{1,2,\ldots,k\}$. Each cluster collects all $n$-permutations whose first $n-1$ elements form the same pattern, that is, these elements in each permutation in the cluster  are order-isomorphic to the same $(n-1)$-permutation. We call such a pattern the {\em signature} of a cluster. See Figure~\ref{clustering-order-3} for the case of $n=3$, and  Figure~\ref{clustering-order-4} for the case of $n=4$ where clusters are thought of as ``super nodes''.  There is exactly one edge associated with each permutation $x_1x_2\cdots x_n$, which goes to the cluster with the signature that is order-isomorphic to $x_2x_3\cdots x_n$. The edges are also viewed as edges between clusters.

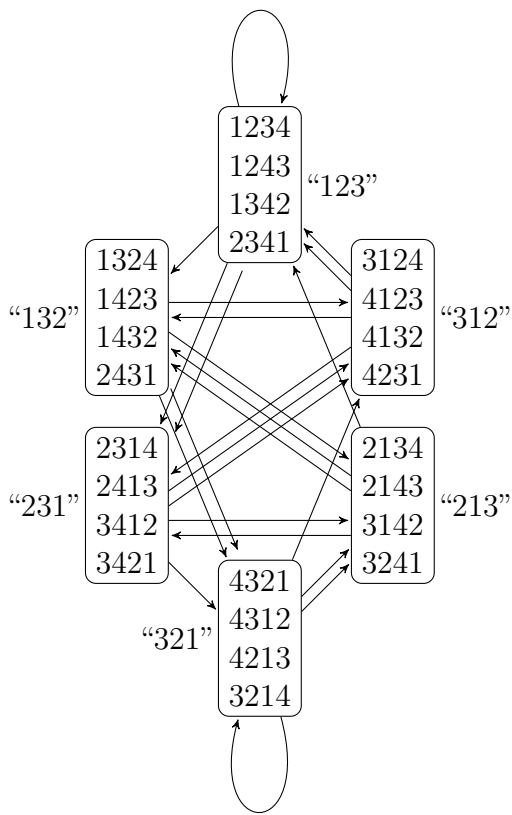
\begin{figure}[h]
\begin{center}
\comm{
\begin{tikzpicture}[->,>=stealth',shorten >=1pt,node distance=2.5cm,auto,main node/.style={rectangle,rounded corners,draw,align=center}]

%\begin{tikzpicture}[->,>=stealth',shorten >=1pt,node distance=1.5cm,auto,main node/.style={circle,draw,align=center}]

\node[main node] (1) {1234 \\ 1243 \\ 1342 \\ 2341}; 
\node (7) [right of=1,node distance=1.1cm] {``123''};
\node[main node] (2) [below left of=1] {1324 \\ 1423 \\ 1432 \\ 2431};
\node (10) [left of=2,node distance=1.1cm] {``132''};
\node[main node] (3) [below right of=1] {3124 \\ 4123 \\ 4132 \\ 4231};
\node (8) [right of=3,node distance=1.1cm] {``312''};
\node[main node] (4) [below of=2] {2314 \\ 2413 \\3412 \\ 3421};
\node (11) [left of=4,node distance=1.1cm] {``231''};
\node[main node] (5) [below of=3] {2134 \\ 2143 \\ 3142 \\ 3241};
\node (9) [right of=5,node distance=1.1cm] {``213''};
\node[main node] (6) [below left of=5] {4321 \\ 4312 \\ 4213 \\ 3214};
\node (12) [left of=6,node distance=1.1cm] {``321''};

\path
(1) edge node {} (2)
     edge node {} (4);

\path [draw,transform canvas={shift={(0.2,-0.1)}}] 
(1) edge node {} (4);
%     edge [bend right=95] node {} (4);

\path [draw,transform canvas={shift={(0,0.2)}}] 
(2) edge node {} (3)
     edge node {} (5);

\path [draw,transform canvas={shift={(0.15,0.1)}}] 
(2) edge node {} (6);

\path     
(2) edge node {} (6);

\path [draw,transform canvas={shift={(0,-0.2)}}] 
(3) edge node {} (1);

\path
(3) edge node {} (1)
     edge node {} (2)
     edge node {} (4);

\path [draw,transform canvas={shift={(0,-0.4)}}] 
(4) edge node {} (3);

\path [draw,transform canvas={shift={(0,-0.2)}}] 
(4) edge node {} (3)
     edge node {} (5)
     edge node {} (6);

\path [draw,transform canvas={shift={(0,-0.4)}}] 
(5) edge node {} (4);

\path [draw,transform canvas={shift={(0,-0.2)}}] 
(5) edge node {} (2);

\path
(5) edge node {} (1)
     edge node {} (2);

\path
(6) edge node {} (5)
     edge node {} (3);

\path [draw,transform canvas={shift={(0,-0.2)}}] 
(6) edge node {} (5);
     
\path
(1) edge [loop above] node {} (1);
\path
(6) edge [loop below] node {} (6);

\end{tikzpicture}}

\vspace{-1cm}
\end{center}
\caption{Clustering the graph of overlapping 4-permutations}\label{clustering-order-4}
\end{figure}

Any Eulerian cycle in a graph formed by clusters can be extended to a Hamiltonian cycle in the graph of overlapping permutations (since each edge corresponds to exactly one permutation and we know this permutation). At least some of these Hamiltonian cycles (possibly all, which is conjectured),  can be extended to u-cycles for permutations via linear extensions of partially ordered sets as described in~\cite{CDG1992}. 

\begin{figure}[h]
\begin{center}
\comm{
\begin{tikzpicture}[->,>=stealth',shorten >=1pt,node distance=2cm,auto,main node/.style={rectangle,rounded corners,draw,align=center}]

\node[main node] (1) {{\tiny $\left(\begin{array}{ccc} 1 & 2 & 3 \\ 1 & 2 & 3 \end{array}\right)$} {\tiny $\left(\begin{array}{ccc} 1 & 3 & 2 \\ 1 & 2 & 3 \end{array}\right)$} {\tiny $\left(\begin{array}{ccc} 2 & 3 & 1 \\ 1 & 2 & 3 \end{array}\right)$} \\[1mm] {\tiny $\left(\begin{array}{ccc} 1 & 2 & 3 \\ 1 & 3 & 2 \end{array}\right)$} {\tiny $\left(\begin{array}{ccc} 1 & 3 & 2 \\ 1 & 3 & 2 \end{array}\right)$} {\tiny $\left(\begin{array}{ccc} 2 & 3 & 1 \\ 1 & 3 & 2 \end{array}\right)$}  \\[1mm]  {\tiny $\left(\begin{array}{ccc} 1 & 2 & 3 \\ 2 & 3 & 1 \end{array}\right)$} {\tiny $\left(\begin{array}{ccc} 1 & 3 & 2 \\ 2 & 3 & 1 \end{array}\right)$} {\tiny $\left(\begin{array}{ccc} 2 & 3 & 1 \\ 2 & 3 & 1 \end{array}\right)$} }; 
\node (2) [above of=1,node distance=1.4cm,xshift=-1.5cm] {``{\tiny $\left(\begin{array}{cc} 1 & 2  \\ 1 & 2  \end{array}\right)$}''};

\node[main node] (3)  [right of=1,xshift=5cm]  {{\tiny $\left(\begin{array}{ccc} 2 & 1 & 3 \\ 2 & 1 & 3 \end{array}\right)$} {\tiny $\left(\begin{array}{ccc} 3 & 1 & 2 \\ 2 & 1 & 3 \end{array}\right)$} {\tiny $\left(\begin{array}{ccc} 3 & 2 & 1 \\ 2 & 1 & 3 \end{array}\right)$} \\[1mm] 
{\tiny $\left(\begin{array}{ccc} 2 & 1 & 3 \\ 3 & 1 & 2 \end{array}\right)$} {\tiny $\left(\begin{array}{ccc} 3 & 1 & 2 \\ 3 & 1 & 2 \end{array}\right)$} {\tiny $\left(\begin{array}{ccc} 3 & 2 & 1 \\ 3 & 1 & 2 \end{array}\right)$}  \\[1mm]  
{\tiny $\left(\begin{array}{ccc} 2 & 1 & 3 \\ 3 & 2 & 1 \end{array}\right)$} {\tiny $\left(\begin{array}{ccc} 3 & 1 & 2 \\ 3 & 2 & 1 \end{array}\right)$} {\tiny $\left(\begin{array}{ccc} 3 & 2 & 1 \\ 3 & 2 & 1 \end{array}\right)$} }; 
\node (4) [above of=3,node distance=1.4cm,xshift=1.5cm] {``{\tiny $\left(\begin{array}{cc} 2 & 1  \\ 2 & 1  \end{array}\right)$}''};

\node[main node] (5)  [below of=1,yshift=-1.2cm]  {{\tiny $\left(\begin{array}{ccc} 1 & 2 & 3 \\ 2 & 1 & 3 \end{array}\right)$} {\tiny $\left(\begin{array}{ccc} 1 & 3 & 2 \\ 2 & 1 & 3 \end{array}\right)$} {\tiny $\left(\begin{array}{ccc} 2 & 3 & 1 \\ 2 & 1 & 3 \end{array}\right)$} \\[1mm] {\tiny $\left(\begin{array}{ccc} 1 & 2 & 3 \\ 3 & 1 & 2 \end{array}\right)$} {\tiny $\left(\begin{array}{ccc} 1 & 3 & 2 \\ 3 & 1 & 2 \end{array}\right)$} {\tiny $\left(\begin{array}{ccc} 2 & 3 & 1 \\ 3 & 1 & 2 \end{array}\right)$}  \\[1mm]  {\tiny $\left(\begin{array}{ccc} 1 & 2 & 3 \\ 3 & 2 & 1 \end{array}\right)$} {\tiny $\left(\begin{array}{ccc} 1 & 3 & 2 \\ 3 & 2 & 1 \end{array}\right)$} {\tiny $\left(\begin{array}{ccc} 2 & 3 & 1 \\ 3 & 2 & 1 \end{array}\right)$} }; 
\node (6) [above of=5,node distance=1.4cm] {``{\tiny $\left(\begin{array}{cc} 1 & 2  \\ 2 & 1  \end{array}\right)$}''};

\node[main node] (7)  [right of=5,xshift=5cm]  {{\tiny $\left(\begin{array}{ccc} 2 & 1 & 3 \\ 1 & 2 & 3 \end{array}\right)$} {\tiny $\left(\begin{array}{ccc} 3 & 1 & 2 \\ 1 & 2 & 3 \end{array}\right)$} {\tiny $\left(\begin{array}{ccc} 3 & 2 & 1 \\ 1 & 2 & 3 \end{array}\right)$} \\[1mm] {\tiny $\left(\begin{array}{ccc} 2 & 1 & 3 \\ 1 & 3 & 2 \end{array}\right)$} {\tiny $\left(\begin{array}{ccc} 3 & 1 & 2 \\ 1 & 3 & 2 \end{array}\right)$} {\tiny $\left(\begin{array}{ccc} 3 & 2 & 1 \\ 1 & 3 & 2 \end{array}\right)$}  \\[1mm]  {\tiny $\left(\begin{array}{ccc} 2 & 1 & 3 \\ 2 & 3 & 1 \end{array}\right)$} {\tiny $\left(\begin{array}{ccc} 3 & 1 & 2 \\ 2 & 3 & 1 \end{array}\right)$} {\tiny $\left(\begin{array}{ccc} 3 & 2 & 1 \\ 2 & 3 & 1 \end{array}\right)$} }; 
\node (8) [above of=7,node distance=1.4cm] {``{\tiny $\left(\begin{array}{cc} 2 & 1  \\ 1 & 2  \end{array}\right)$}''};

\path[draw,transform canvas={shift={(0,0.3)}}] 
(1) edge node {} (3);
\path[draw,transform canvas={shift={(0,0.4)}}] 
(1) edge node {} (3);
\path[draw,transform canvas={shift={(0,0.5)}}] 
(1) edge node {} (3);
\path[draw,transform canvas={shift={(0,0.6)}}] 
(1) edge node {} (3);

\path[draw,transform canvas={shift={(0,-0.3)}}] 
(3) edge node {} (1);
\path[draw,transform canvas={shift={(0,-0.4)}}] 
(3) edge node {} (1);
\path[draw,transform canvas={shift={(0,-0.5)}}] 
(3) edge node {} (1);
\path[draw,transform canvas={shift={(0,-0.6)}}] 
(3) edge node {} (1);

\path[draw,transform canvas={shift={(0,0.3)}}] 
(5) edge node {} (7);
\path[draw,transform canvas={shift={(0,0.4)}}] 
(5) edge node {} (7);
\path[draw,transform canvas={shift={(0,0.5)}}] 
(5) edge node {} (7);
\path[draw,transform canvas={shift={(0,0.6)}}] 
(5) edge node {} (7);

\path[draw,transform canvas={shift={(0,-0.3)}}] 
(7) edge node {} (5);
\path[draw,transform canvas={shift={(0,-0.4)}}] 
(7) edge node {} (5);
\path[draw,transform canvas={shift={(0,-0.5)}}] 
(7) edge node {} (5);
\path[draw,transform canvas={shift={(0,-0.6)}}] 
(7) edge node {} (5);

\path[draw,transform canvas={shift={(0.5,0)}}] 
(1) edge node {} (7);
\path[draw,transform canvas={shift={(0.7,0)}}] 
(1) edge node {} (7);

\path[draw,transform canvas={shift={(-0.5,0)}}] 
(7) edge node {} (1);
\path[draw,transform canvas={shift={(-0.7,0)}}] 
(7) edge node {} (1);

\path[draw,transform canvas={shift={(0.5,0)}}] 
(3) edge node {} (5);
\path[draw,transform canvas={shift={(0.7,0)}}] 
(3) edge node {} (5);

\path[draw,transform canvas={shift={(-0.5,0)}}] 
(5) edge node {} (3);
\path[draw,transform canvas={shift={(-0.7,0)}}] 
(5) edge node {} (3);

\path[draw,transform canvas={shift={(-1.8,0)}}] 
(1) edge node {} (5);
\path[draw,transform canvas={shift={(-1.9,0)}}] 
(1) edge node {} (5);

\path[draw,transform canvas={shift={(-2.2,0)}}] 
(5) edge node {} (1);
\path[draw,transform canvas={shift={(-2.3,0)}}] 
(5) edge node {} (1);

\path[draw,transform canvas={shift={(1.8,0)}}] 
(7) edge node {} (3);
\path[draw,transform canvas={shift={(1.9,0)}}] 
(7) edge node {} (3);

\path[draw,transform canvas={shift={(2.2,0)}}] 
(3) edge node {} (7);
\path[draw,transform canvas={shift={(2.3,0)}}] 
(3) edge node {} (7);

\path
(1) edge [loop above] node {} (1);
\path
(3) edge [loop above] node {} (3);
\path
(5) edge [loop below] node {} (5);
\path
(7) edge [loop below] node {} (7);

\end{tikzpicture}}

\vspace{-1cm}
\end{center}
\caption{Clustering the graph of overlapping 3-dimensional 3-permutations}\label{clustering-3-3}
\end{figure}
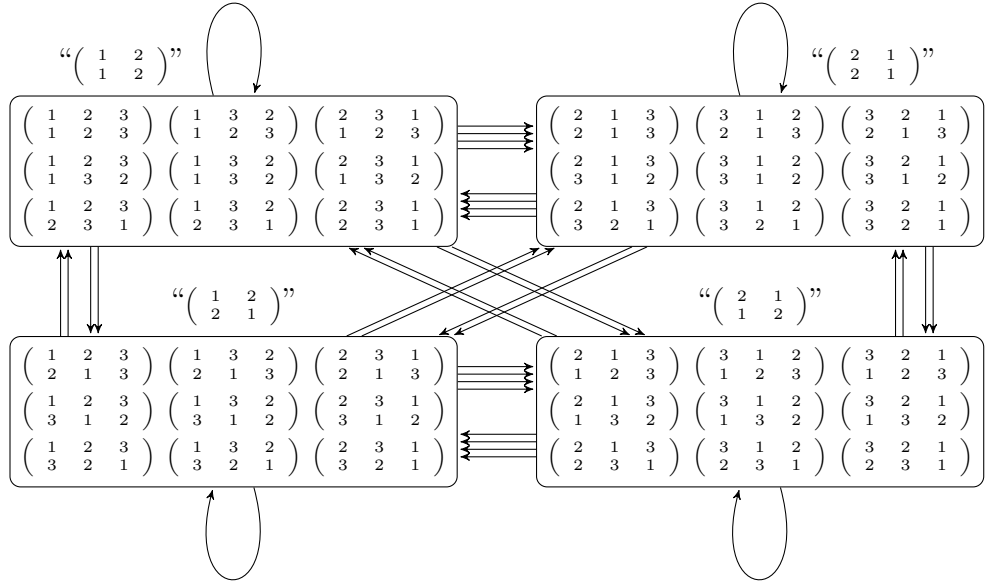

It is straightforward to extend the just introduced notions, namely those of the graph of overlapping permutations, cluster and signature to the case of $d$-dimensional $n$-permutations. In particular, there will be $((n-1)!)^{d-1}$ clusters and the signature of a cluster is a $(d-1)\times (n-1)$ matrix whose rows are ``usual'' permutations. Clearly, for $d$-dimensional 2-permutations we have only one cluster. Finally, clustering of the graph of overlapping 3-dimensional 3-permutations is presented in Figure~\ref{clustering-3-3}.

\section{Proof of Theorem~\ref{main-thm}}\label{proof-sec}

Thinking of $d$-dimensional $n$-permutations being represented by $d-1$ rows, such a permutation is of {\em type $i$}, $i>0$, if $i$ of its rows are of the form $x_1x_2\ldots x_n$ with $|x_n-x_1|=1$ and the remaining rows are monotone (that is, either increasing, $12\ldots n$, or decreasing, $n(n-1)\ldots 1$). To ensure  unambiguity of the definition, we assume that $n\geq 3$, but in either case, our definition generalizes the notion of {\em twins} in \cite{KPV}. For example, the 5-dimensional 4-permutation {\footnotesize $\left(\begin{array}{cccc} 4 & 3 & 2 & 1  \\ 2 & 1 & 4 & 3  \\ 1 & 4 & 3 & 2 \\ 1 & 2 & 3 & 4  \end{array}\right)$} is of type 2, where the top and bottom rows are monotone. If a permutation is not of type $i$ for some $i>0$, we say it is of type $0$.

Two distinct $d$-dimensional $n$-permutations are {\em $i$-twins} if both of them are of type $i$, the monotone rows in them match each other (in position and type) and the permutations belong to the same cluster (i.e.\ both permutations have the same signature, the reduced forms of the leftmost $n-1$ columns). For example,  {\footnotesize $\left(\begin{array}{ccc} 3 & 2 & 1  \\ 2 & 1 & 3  \\ 1 & 3 & 2 \\ 1 & 2 & 3  \end{array}\right)$} and  {\footnotesize $\left(\begin{array}{ccc} 3 & 2 & 1  \\ 2 & 1 & 3  \\ 2 & 3 & 1 \\ 1 & 2 & 3  \end{array}\right)$}, as well as {\footnotesize $\left(\begin{array}{ccc} 1 & 3 & 2 \\ 1 & 3 & 2 \end{array}\right)$}, {\footnotesize $\left(\begin{array}{ccc} 1 & 3 & 2 \\ 2 & 3 & 1 \end{array}\right)$}, {\footnotesize $\left(\begin{array}{ccc} 2 & 3 & 1 \\ 1 & 3 & 2 \end{array}\right)$} and {\footnotesize $\left(\begin{array}{ccc} 2 & 3 & 1 \\ 2 & 3 & 1 \end{array}\right)$}, are $2$-twins. For another example,  {\footnotesize $\left(\begin{array}{ccc} 1 & 2 & 3 \\ 1 & 3 & 2 \end{array}\right)$} and  {\footnotesize $\left(\begin{array}{ccc} 1 & 2 & 3 \\ 2 & 3 & 1 \end{array}\right)$} are 1-twins. Finally, examples of (4-dimensional) 3-twins are  {\footnotesize $\left(\begin{array}{cccc} 3 & 1 & 2 & 4 \\ 2 & 4 & 1 & 3 \\ 2 & 3 & 4 & 1 \end{array}\right)$} and  {\footnotesize $\left(\begin{array}{cccc} 4 & 1 & 2 & 3 \\ 3 & 4 & 1 & 2 \\ 1 & 3 & 4 & 2 \end{array}\right)$}. 

Note that the same cluster can have more than one class of $i$-twins for a fixed $i$, where an $i$-twin class is an equivalence class under the equivalence relation of being $i$-twins.  However, while each cluster has one $(d-1)$-twin class (see Lemma~\ref{lem1}), for $n\geq 4$ some clusters do not have an $i$-twin class for some values of  $i\neq d-1$. 
%(Figure~\ref{clustering-3-3} shows that this is not true for $n=3$, since there are two distinct 1-twin classes in each cluster then). 
Also note that by definition, any $i$-twin class is contained within a cluster. 

We invite the Reader to verify the statements of the following results in Figure~\ref{clustering-3-3}, the largest clustered graph of overlapping permutations that is feasible to draw in full in this paper. 

\begin{lem}\label{lem1}  Each cluster contains exactly $2^{d-1}$ $(d-1)$-twins within one $(d-1)$-twin class. Moreover, for a permutation of type $i>0$, the size of its $i$-twin class is $2^i$. \end{lem}

\begin{proof} For the first claim, let the signature of a cluster  be ``$S=\left\{s^j_m\right\}_{\begin{subarray}{l} 1 \leq j\leq d \\ 1\le m\le n-1\end{subarray}}$''. The only possibilities to create $(d-1)$-twins are to adjoin, independently in each row $j$, $2\leq j\leq d$, $(s^j_1)^{+}$ or $(s^j_1)^{-}$ at the end of $s^j_1\ldots s^j_{n-1}$, which results in $2^{d-1}$ options. The second claim follows from  essentially the same arguments keeping in mind that the monotone rows are fixed (so that we cannot have $i$-twins with different sets/placements of monotone rows).\end{proof}

{\em Parallel edges} between two clusters are multiple edges oriented in the same way. 

\begin{lem}\label{lem2} For $i>0$, an $i$-twin class $I$ in cluster $X$ contributes $2^i$ parallel edges from $X$ to some cluster $Y$, and there are no other edges from $X$ to $Y$. Also, if there are $k$ (disjoint) $i$-twin classes $I_1$, $I_2,\ldots,I_k$ in cluster $X$ then there are distinct clusters $Z_1$, $Z_2,\ldots,Z_k$, each having $i$-twin classes with $2^i$ parallel edges from $Z_j$ to $X$, $j=1,2,\ldots,k$. \end{lem}
%\begin{lem}\label{lem2} We have that $i$-twins in cluster $X$, if any, forming a class of permutations $I$, contribute $2^i$ parallel edges from $X$ to a cluster $Y$, and there are no other edges from $X$ to $Y$. Also, if there are $k$ (disjoint) classes, $I_1$, $I_2,\ldots,I_k$, of $i$-twins in cluster $X$ then there are distinct clusters $Z_1$, $Z_2,\ldots,Z_k$, each having $i$-twins, with $2^i$ parallel edges from $Z_j$ to $X$, $j=1,2,\ldots,k$. \end{lem}

\begin{proof} By definition, all $i$-twins in the set $I$ have the same reduced form of the $n-1$ rightmost columns, since the leftmost and rightmost entries have the same relative order to the middle $n-2$ entries. Hence, each $i$-twin contributes an edge oriented towards the same cluster, say $Y$. By Lemma~\ref{lem1}, there are $2^i$ $i$-twins which completes the proof of the first claim, because edges (equivalently, $d$-dimensional $n$-permutations) starting in the same cluster and ending in the same cluster have, within each row, the same relative ordering of the first $n-1$ elements and the same relative ordering of the last $n-1$ elements, which determines the relative ordering of all pairs of elements except for the first and last. 

Now, assume, w.l.o.g., that the $i$ non-monotone rows in a class $I_j$ are the top rows, and for $m=1,\ldots,i$, the $m$-th row in an $i$-twin is $s^m_1\ldots s^m_n$. There are $2^i$ parallel edges coming to $X$ from the cluster $Z_j$ having $i$-twins with the top $i$ rows $s^m_0s^m_1\ldots s^m_{n-1}$, $m=1,\ldots,i$,  where $s^m_0\in\{(s^m_{n-1})^+,(s^m_{n-1})^-\}$, and the monotone rows matching the respective rows in the class $I_j$. Note that $\red(s^m_0s^m_1\ldots s^m_{n-1})$ is the same regardless of which $i$-twin in $I_j$ was chosen to define $s^m_1\ldots s^m_n$. Finally, the only way $I_a\neq I_b$ can occur within  the same cluster $X$ is if the monotone rows differ. In such a case, $Z_a\neq Z_b$ necessarily follows. 
\end{proof}

\begin{lem}\label{lem3}  The clustered graph of overlapping permutations can be partitioned into a disjoint union of cycles of length $n-1$ formed by $2^{d-1}$ parallel edges between clusters (any other edges are to be ignored). Also, for each $i$, $1\leq i\leq d-1$, there are  
\begin{equation}\label{num-cycles}
\frac{1}{n-1}{d-1\choose i}((n-1)!)^i2^{d-i-1}
\end{equation} 
disjoint cycles of length $n-1$ formed by $2^i$ parallel edges between clusters. \end{lem}

\begin{proof} First note that from Lemma~\ref{lem1} (first part) and Lemma~\ref{lem2}, we can partition the clustered graph of overlapping permutations   into a disjoint union of cycles formed by  $2^{d-1}$ parallel edges, and  there are disjoint cycles of $2^i$ parallel edges formed by $i$-twins. Next, observe that each cycle formed by $2^i$ parallel edges must be of length $n-1$. This follows from the observation that if there are $2^i$ parallel edges from a cluster $X$ to a cluster $Y$ then the signature (of length $n-1$) of $Y$ restricted to non-monotone rows is the cyclic shift of the signature of $X$ restricted to non-monotone rows one position to the right, while the monotone rows stay the same. Hence, only $n-1$ distinct shifts will take place. 

Finally, for computing the number of disjoint cycles formed by $2^i$ parallel edges, we  first note that there are  ${d-1\choose i}((n-1)!)^i2^{d-1-i}$ possible ways to choose an $i$-twin class to start such a cycle, where ${d-1\choose i}2^{d-1-i}$ is the number of ways to select monotone rows and their types, and $((n-1)!)^i$ ways is the number of ways to pick the signature of non-monotone rows. However, because each cycle is of length $n-1$ we need to divide by this number (otherwise, we count each cycle $n-1$ times since it can begin in any permutation of the cycle). 
\end{proof}

To complete our proof of Theorem~\ref{main-thm} we  follow the following observations. Note that each edge in the clustered graph of overlapping permutations corresponds to a $d$-dimensional permutation. Since u-cycles for $d$-dimensional permutations exist \cite{KQ}, the clustered graph must contain an Eulerian cycle.  Consider an Eulerian path in this graph, which gives a u-word of length $(n!)^{d-1}+n-1$. For any fixed $i$, $1\leq i\leq d-1$, pick a cycle of length $n-1$ formed by $2^i$ parallel edges (that exists by Lemma~\ref{lem2}). Decide on a row $j$ of each vertex of the cycle to make the first element be equal to the last element, that is, to be of the form $s_1s_2\ldots s_{n-1}s_1$. Using a repeated element in each row $j$ effectively removes $n-1$ edges from the cycle and reduces the number of permutations in the graph by $n-1$. The graph remains balanced and connected and hence an Eulerian cycle still exists that corresponds to a u-word of length $(n!)^{d-1}+n-1-(n-1)$. We can now chose another (arbitrary!) cycle to remove that would result in a u-word of length $(n!)^{d-1}+n-1-2(n-1)$. And so on. The maximum compression is achieved by removing all multiple edges -- one from each set of parallel edges and $n-1$ from each cycle -- where, for $(n-1)$-cycles comprising $2^i$ parallel edges, this requires ``removing a cycle'' $2^i-1$ times (leaving 1 edge per set), resulting in the total number of applications being calculated by multiplying the formula in Lemma~\ref{lem3} by $2^i-1$ and summing over all $i$:
%The maximum compression is achieved once all multiple edges are removed (one from any set of parallel edges, and $n-1$ from each cycle at a time), resulting in the number of applications of the ``removing a cycle'' be given by multiplying the formula in Lemma~\ref{lem3} by $(2^i-1)$ and summing over all possible $i$:
$$\frac{2^{d-1}}{n-1}\sum_{i=1}^{d-1}{d-1\choose i}((n-1)!)^i(1-2^{-i})= \frac{2^{d-1}}{n-1}\left[(1+(n-1)!)^{d-1}-\left(1+\frac{(n-1)!}{2}\right)^{d-1}\right].$$

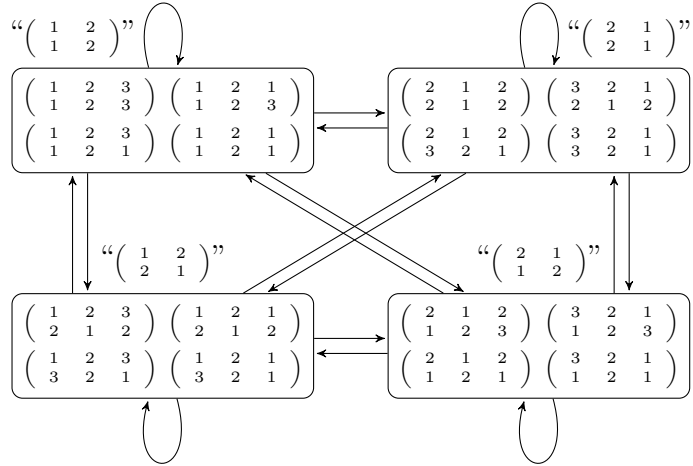
\begin{figure}[h]
\begin{center}
\comm{
\begin{tikzpicture}[->,>=stealth',shorten >=1pt,node distance=2cm,auto,main node/.style={rectangle,rounded corners,draw,align=center}]

\node[main node] (1) {{\tiny $\left(\begin{array}{ccc} 1 & 2 & 3 \\ 1 & 2 & 3 \end{array}\right)$} {\tiny $\left(\begin{array}{ccc} 1 & 2 & 1 \\ 1 & 2 & 3 \end{array}\right)$} \\[1mm] {\tiny $\left(\begin{array}{ccc} 1 & 2 & 3 \\ 1 & 2 & 1 \end{array}\right)$} {\tiny $\left(\begin{array}{ccc} 1 & 2 & 1 \\ 1 & 2 & 1 \end{array}\right)$}}; 
\node (2) [above of=1,node distance=1.1cm,xshift=-1.2cm] {``{\tiny $\left(\begin{array}{cc} 1 & 2  \\ 1 & 2  \end{array}\right)$}''};

\node[main node] (3)  [right of=1,xshift=3cm]  { {\tiny $\left(\begin{array}{ccc} 2 & 1 & 2 \\ 2 & 1 & 2 \end{array}\right)$} {\tiny $\left(\begin{array}{ccc} 3 & 2 & 1 \\ 2 & 1 & 2 \end{array}\right)$} \\[1mm] 
{\tiny $\left(\begin{array}{ccc} 2 & 1 & 2 \\ 3 & 2 & 1 \end{array}\right)$} {\tiny $\left(\begin{array}{ccc} 3 & 2 & 1 \\ 3 & 2 & 1 \end{array}\right)$} }; 
\node (4) [above of=3,node distance=1.1cm,xshift=1.2cm] {``{\tiny $\left(\begin{array}{cc} 2 & 1  \\ 2 & 1  \end{array}\right)$}''};

\node[main node] (5)  [below of=1,yshift=-1cm]  {{\tiny $\left(\begin{array}{ccc} 1 & 2 & 3 \\ 2 & 1 & 2 \end{array}\right)$} {\tiny $\left(\begin{array}{ccc} 1 & 2 & 1 \\ 2 & 1 & 2 \end{array}\right)$} \\[1mm] {\tiny $\left(\begin{array}{ccc} 1 & 2 & 3 \\  3& 2 & 1 \end{array}\right)$}  {\tiny $\left(\begin{array}{ccc} 1 & 2 & 1 \\ 3 & 2 & 1 \end{array}\right)$} }; 
\node (6) [above of=5,node distance=1.1cm] {``{\tiny $\left(\begin{array}{cc} 1 & 2  \\ 2 & 1  \end{array}\right)$}''};

\node[main node] (7)  [right of=5,xshift=3cm]  {{\tiny $\left(\begin{array}{ccc} 2 & 1 & 2 \\ 1 & 2 & 3 \end{array}\right)$} {\tiny $\left(\begin{array}{ccc} 3 & 2 & 1 \\ 1 & 2 & 3 \end{array}\right)$} \\[1mm] {\tiny $\left(\begin{array}{ccc} 2 & 1 & 2 \\ 1 & 2 & 1 \end{array}\right)$} {\tiny $\left(\begin{array}{ccc} 3 & 2 & 1 \\ 1 & 2 & 1 \end{array}\right)$} }; 
\node (8) [above of=7,node distance=1.1cm] {``{\tiny $\left(\begin{array}{cc} 2 & 1  \\ 1 & 2  \end{array}\right)$}''};

%horizontal lines
\path[draw,transform canvas={shift={(0,0.1)}}] 
(1) edge node {} (3);
\path[draw,transform canvas={shift={(0,-0.1)}}] 
(3) edge node {} (1);
\path[draw,transform canvas={shift={(0,0.1)}}] 
(5) edge node {} (7);
\path[draw,transform canvas={shift={(0,-0.1)}}] 
(7) edge node {} (5);

%diagonal lines
\path[draw,transform canvas={shift={(0.2,0)}}] 
(1) edge node {} (7);
\path[draw,transform canvas={shift={(-0.1,0)}}] 
(7) edge node {} (1);
\path[draw,transform canvas={shift={(0.2,0)}}] 
(3) edge node {} (5);
\path[draw,transform canvas={shift={(-0.1,0)}}] 
(5) edge node {} (3);

%vertical
\path[draw,transform canvas={shift={(-1,0)}}] 
(1) edge node {} (5);
\path[draw,transform canvas={shift={(-1.2,0)}}] 
(5) edge node {} (1);
\path[draw,transform canvas={shift={(1,0)}}] 
(7) edge node {} (3);
\path[draw,transform canvas={shift={(1.2,0)}}] 
(3) edge node {} (7);

\path
(1) edge [loop above] node {} (1);
\path
(3) edge [loop above] node {} (3);
\path
(5) edge [loop below] node {} (5);
\path
(7) edge [loop below] node {} (7);

\end{tikzpicture}}

\vspace{-1cm}
\end{center}
\caption{The clustered graph giving the maximum compression possibility using incomparable elements at distance 2 for 3-dimensional 3-permutations}\label{max-compression-3-3}
\end{figure}

Figure~\ref{max-compression-3-3} shows the graph corresponding to the maximum compression possibility using incomparable elements at distance 2 for 3-dimensional 3-permutations. A particular Eulerian path in the graph goes as follows:

\noindent
 {\footnotesize$\left(\begin{array}{ccc} 1 & 2 & 3 \\ 1 & 2 & 3   \end{array}\right)$} $\rightarrow$
 {\footnotesize$\left(\begin{array}{ccc} 1 & 2 & 1 \\ 1 & 2 & 1   \end{array}\right)$} $\rightarrow$
 {\footnotesize$\left(\begin{array}{ccc} 3 & 2 & 1 \\ 3 & 2 & 1   \end{array}\right)$} $\rightarrow$
 {\footnotesize$\left(\begin{array}{ccc} 3 & 2 & 1 \\ 2 & 1 & 2   \end{array}\right)$} $\rightarrow$
 {\footnotesize$\left(\begin{array}{ccc} 3 & 2 & 1 \\ 1 & 2 & 3   \end{array}\right)$} $\rightarrow$
 {\footnotesize$\left(\begin{array}{ccc} 2 & 1 & 2 \\ 1 & 2 & 1   \end{array}\right)$} $\rightarrow$
 {\footnotesize$\left(\begin{array}{ccc} 1 & 2 & 3 \\ 3 & 2 & 1   \end{array}\right)$} $\rightarrow$
 {\footnotesize$\left(\begin{array}{ccc} 1 & 2 & 3 \\ 2 & 1 & 2   \end{array}\right)$} $\rightarrow$
{\footnotesize$\left(\begin{array}{ccc} 1 & 2 & 1 \\ 1 & 2 & 3   \end{array}\right)$} $\rightarrow$
{\footnotesize$\left(\begin{array}{ccc} 2 & 1 & 2 \\ 1 & 2 & 3   \end{array}\right)$} $\rightarrow$
{\footnotesize$\left(\begin{array}{ccc} 1 & 2 & 3 \\ 1 & 2 & 1   \end{array}\right)$} $\rightarrow$
{\footnotesize$\left(\begin{array}{ccc} 1 & 2 & 1 \\ 3 & 2 & 1   \end{array}\right)$} $\rightarrow$
{\footnotesize$\left(\begin{array}{ccc} 2 & 1 & 2 \\ 3 & 2 & 1   \end{array}\right)$} $\rightarrow$
{\footnotesize$\left(\begin{array}{ccc} 1 & 2 & 1 \\ 2 & 1 & 2   \end{array}\right)$} $\rightarrow$
{\footnotesize$\left(\begin{array}{ccc} 3 & 2 & 1 \\ 1 & 2 & 1   \end{array}\right)$} $\rightarrow$
{\footnotesize$\left(\begin{array}{ccc} 2 & 1 & 2 \\ 2 & 1 & 2   \end{array}\right)$}.

\noindent
This path can be transformed into a u-word of length $18$ for $3$-dimensional $3$-permutations, as given below, using the following general description.

Consider the $(k+1)$-st step $A \rightarrow B$, involving the matrices  
$A = \left\{ a_{i,j} \right\}_{\substack{1 \le i \le d-1 \\ 1 \le j \le n}}$  
and  
$B = \left\{ b_{i,j} \right\}_{\substack{1 \le i \le d-1 \\ 1 \le j \le n}}$  
in an Eulerian path, and assume that the matrix  
$X_k = \left\{ x_{i,j} \right\}_{\substack{1 \le i \le d-1 \\ 1 \le j \le k+n}}$  
has already been constructed (with $X_0$ being the initial matrix in the Eulerian path). Note that $\red(x_{i,k+1}x_{i,k+2}\ldots x_{i,k+n})=a_{i,1}a_{i,2}\ldots a_{i,n}$.

To construct $X_{k+1}$, we proceed independently in each row  
$x_{i,1}x_{i,2}\ldots x_{i,n+k}$, $1 \leq i \leq d-1$,  
where $\max_i$ denotes the largest element in row $i$ of $X_k$:

\begin{itemize}
\item If possible, let $x_{i,n+k+1}$ be the smallest option in  
$\{x_{i,1},\ldots,x_{i,n+k},1+\max_i\}$ such that  
$\red(x_{i,k+2}x_{i,k+3}\ldots x_{i,k+n+1}) = b_{i,1}b_{i,2}\ldots b_{i,n}$,  
and do not change the other elements in row $i$ of $X_k$.  
This includes the case when $b_{i,n} = n$.

\item Otherwise, we have $b_{i,n} < n$ and let $x_{i,n+k+1}$ be the $b_{i,n}$-th smallest element of  
$\{x_{i,k+2},x_{i,k+3},\ldots,x_{i,k+n}\}$.  
In addition, for $1 \leq j \leq n+k$, replace $x_{i,j}$ with $x_{i,j}+1$ whenever $x_{i,j} \geq x_{i,n+k+1}$.
\end{itemize}
For the Eulerian path above,  the steps of constructing the u-word are as follows, where we   
indicate only those arrows that require rewriting existing elements in one of the rows:
\begin{equation*}
{\footnotesize
\left(\begin{array}{ccccc}
 1& 2 & 3 & 2 & 1\\
1 & 2 & 3 & 2 & 1
\end{array}\right)
\rightarrow
\left(\begin{array}{cccccc}
2 & 3 & 4 & 3 & 2 & 1 \\
1 & 2 & 3 & 2 & 1 & 2
\end{array}\right)
\rightarrow
\left(\begin{array}{cccccccccccccc}
3 & 4 & 5 & 4 & 3 & 2 & 1 & 2 & 3 & 4 & 3 & 4 & 5 & 4 \\
1 & 2 & 3 & 2 & 1 & 2 & 3 & 2 & 1 & 2 & 3 & 4 & 3 & 1
\end{array}\right)
}
\end{equation*}
\begin{equation*}
\rightarrow {\footnotesize
\left(\begin{array}{cccccccccccccccccc}
3 & 4 & 5 & 4 & 3 & 2 & 1 & 2 & 3 & 4 & 3 & 4 & 5 & 4 & 5 & 4 & 1 & 4 \\
2 & 3 & 4 & 3 & 2 & 3 & 4 & 3 & 2 & 3 & 4 & 5 & 4 & 2 & 1 & 2 & 1 & 2
\end{array}\right).
}
\end{equation*}
%\begin{equation}\label{example-max-3-3}
% {\footnotesize\left(\begin{array}{cccccccccccccccccc} 3 & 4 & 5 & 4 & 3 & 2 & 1 & 2 & 3 & 4 & 3 & 4 & 5 & 4 & 5 & 4 & 3 & 4  \\  
% 3 & 4 & 5 & 4 & 3 & 4 & 5 & 4 & 3 & 4 & 5 & 6 & 5 & 4 & 3 & 4 & 3 & 4 \end{array}\right)}.
%\end{equation}
 
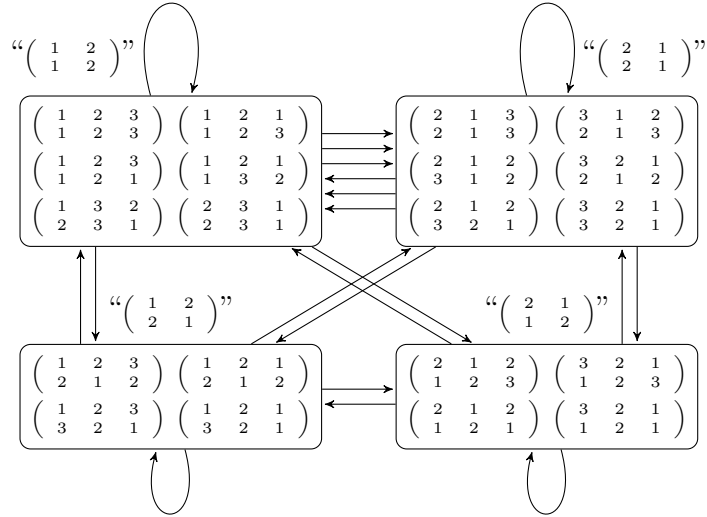
\begin{figure}[h]
\begin{center}
\comm{
\begin{tikzpicture}[->,>=stealth',shorten >=1pt,node distance=2cm,auto,main node/.style={rectangle,rounded corners,draw,align=center}]

\node[main node] (1) {{\tiny $\left(\begin{array}{ccc} 1 & 2 & 3 \\ 1 & 2 & 3 \end{array}\right)$} {\tiny $\left(\begin{array}{ccc} 1 & 2 & 1 \\ 1 & 2 & 3 \end{array}\right)$} \\[1mm] {\tiny $\left(\begin{array}{ccc} 1 & 2 & 3 \\ 1 & 2 & 1 \end{array}\right)$} {\tiny $\left(\begin{array}{ccc} 1 & 2 & 1 \\ 1 & 3 & 2 \end{array}\right)$}  \\[1mm] {\tiny $\left(\begin{array}{ccc} 1 & 3 & 2 \\ 2 & 3 & 1 \end{array}\right)$} {\tiny $\left(\begin{array}{ccc} 2 & 3 & 1 \\ 2 & 3 & 1 \end{array}\right)$}}; 
\node (2) [above of=1,node distance=1.5cm,xshift=-1.3cm] {``{\tiny $\left(\begin{array}{cc} 1 & 2  \\ 1 & 2  \end{array}\right)$}''};

\node[main node] (3)  [right of=1,xshift=3cm]  { {\tiny $\left(\begin{array}{ccc} 2 & 1 & 3 \\ 2 & 1 & 3 \end{array}\right)$} {\tiny $\left(\begin{array}{ccc} 3 & 1 & 2 \\ 2 & 1 & 3 \end{array}\right)$} \\[1mm]  {\tiny $\left(\begin{array}{ccc} 2 & 1 & 2 \\ 3 & 1 & 2 \end{array}\right)$} {\tiny $\left(\begin{array}{ccc} 3 & 2 & 1 \\ 2 & 1 & 2 \end{array}\right)$} \\[1mm] 
{\tiny $\left(\begin{array}{ccc} 2 & 1 & 2 \\ 3 & 2 & 1 \end{array}\right)$} {\tiny $\left(\begin{array}{ccc} 3 & 2 & 1 \\ 3 & 2 & 1 \end{array}\right)$} }; 
\node (4) [above of=3,node distance=1.5cm,xshift=1.3cm] {``{\tiny $\left(\begin{array}{cc} 2 & 1  \\ 2 & 1  \end{array}\right)$}''};

\node[main node] (5)  [below of=1,yshift=-1cm]  {{\tiny $\left(\begin{array}{ccc} 1 & 2 & 3 \\ 2 & 1 & 2 \end{array}\right)$} {\tiny $\left(\begin{array}{ccc} 1 & 2 & 1 \\ 2 & 1 & 2 \end{array}\right)$} \\[1mm] {\tiny $\left(\begin{array}{ccc} 1 & 2 & 3 \\  3& 2 & 1 \end{array}\right)$}  {\tiny $\left(\begin{array}{ccc} 1 & 2 & 1 \\ 3 & 2 & 1 \end{array}\right)$} }; 
\node (6) [above of=5,node distance=1.1cm] {``{\tiny $\left(\begin{array}{cc} 1 & 2  \\ 2 & 1  \end{array}\right)$}''};

\node[main node] (7)  [right of=5,xshift=3cm]  {{\tiny $\left(\begin{array}{ccc} 2 & 1 & 2 \\ 1 & 2 & 3 \end{array}\right)$} {\tiny $\left(\begin{array}{ccc} 3 & 2 & 1 \\ 1 & 2 & 3 \end{array}\right)$} \\[1mm] {\tiny $\left(\begin{array}{ccc} 2 & 1 & 2 \\ 1 & 2 & 1 \end{array}\right)$} {\tiny $\left(\begin{array}{ccc} 3 & 2 & 1 \\ 1 & 2 & 1 \end{array}\right)$} }; 
\node (8) [above of=7,node distance=1.1cm] {``{\tiny $\left(\begin{array}{cc} 2 & 1  \\ 1 & 2  \end{array}\right)$}''};

%horizontal lines
\path[draw,transform canvas={shift={(0,0.5)}}] 
(1) edge node {} (3);
\path[draw,transform canvas={shift={(0,0.3)}}] 
(1) edge node {} (3);
\path[draw,transform canvas={shift={(0,0.1)}}] 
(1) edge node {} (3);
\path[draw,transform canvas={shift={(0,-0.1)}}] 
(3) edge node {} (1);
\path[draw,transform canvas={shift={(0,-0.3)}}] 
(3) edge node {} (1);
\path[draw,transform canvas={shift={(0,-0.5)}}] 
(3) edge node {} (1);
\path[draw,transform canvas={shift={(0,0.1)}}] 
(5) edge node {} (7);
\path[draw,transform canvas={shift={(0,-0.1)}}] 
(7) edge node {} (5);

%diagonal lines
\path[draw,transform canvas={shift={(0.2,0)}}] 
(1) edge node {} (7);
\path[draw,transform canvas={shift={(-0.1,0)}}] 
(7) edge node {} (1);
\path[draw,transform canvas={shift={(0.2,0)}}] 
(3) edge node {} (5);
\path[draw,transform canvas={shift={(-0.1,0)}}] 
(5) edge node {} (3);

%vertical
\path[draw,transform canvas={shift={(-1,0)}}] 
(1) edge node {} (5);
\path[draw,transform canvas={shift={(-1.2,0)}}] 
(5) edge node {} (1);
\path[draw,transform canvas={shift={(1,0)}}] 
(7) edge node {} (3);
\path[draw,transform canvas={shift={(1.2,0)}}] 
(3) edge node {} (7);

\path
(1) edge [loop above] node {} (1);
\path
(3) edge [loop above] node {} (3);
\path
(5) edge [loop below] node {} (5);
\path
(7) edge [loop below] node {} (7);

\end{tikzpicture}}

\vspace{-1cm}
\end{center}
\caption{The clustered graph giving a non-maximum compression possibility using incomparable elements at distance 2 for 3-dimensional 3-permutations}\label{non-max-compression-3-3}
\end{figure}

To conclude this section, we present an example of a u-word obtained by a non-maximum compression corresponding to the clustered graph in Figure~\ref{non-max-compression-3-3}. In that figure, only one out of four parallel edges between the top clusters was removed resulting in {\footnotesize $\left(\begin{array}{ccc} 1 & 3 & 2 \\ 2 & 3 & 1 \end{array}\right)$} and {\footnotesize $\left(\begin{array}{ccc} 2 & 3 & 1 \\ 2 & 3 & 1 \end{array}\right)$} being not merged with each other and with  {\footnotesize $\left(\begin{array}{ccc} 1 & 2 & 1 \\ 1 & 3 & 2 \end{array}\right)$}, and {\footnotesize $\left(\begin{array}{ccc} 2 & 1 & 3 \\ 2 & 1 & 3 \end{array}\right)$} and {\footnotesize $\left(\begin{array}{ccc} 3 & 1 & 2 \\ 2 & 1 & 3 \end{array}\right)$} being not merged with each other and with  {\footnotesize $\left(\begin{array}{ccc} 2 & 1 & 2 \\ 3 & 2 & 1 \end{array}\right)$}. A particular Eulerian path in that graph goes as follows:

\noindent
 {\footnotesize$\left(\begin{array}{ccc} 1 & 2 & 3 \\ 1 & 2 & 3   \end{array}\right)$} $\rightarrow$
 {\footnotesize$\left(\begin{array}{ccc} 1 & 2 & 1 \\ 1 & 3 & 2   \end{array}\right)$} $\rightarrow$
 {\footnotesize$\left(\begin{array}{ccc} 3 & 2 & 1 \\ 3 & 2 & 1   \end{array}\right)$} $\rightarrow$
 {\footnotesize$\left(\begin{array}{ccc} 2 & 1 & 2 \\ 3 & 1 & 2   \end{array}\right)$} $\rightarrow$
 {\footnotesize$\left(\begin{array}{ccc} 1 & 3 & 2 \\ 2 & 3 & 1   \end{array}\right)$} $\rightarrow$
 {\footnotesize$\left(\begin{array}{ccc} 2 & 1 & 3 \\ 2 & 1 & 3   \end{array}\right)$} $\rightarrow$
 {\footnotesize$\left(\begin{array}{ccc} 2 & 3 & 1 \\ 2 & 3 & 1   \end{array}\right)$} $\rightarrow$
 {\footnotesize$\left(\begin{array}{ccc} 3 & 2 & 1 \\ 2 & 1 & 2   \end{array}\right)$} $\rightarrow$
{\footnotesize$\left(\begin{array}{ccc} 3 & 2 & 1 \\ 1 & 2 & 3   \end{array}\right)$} $\rightarrow$
{\footnotesize$\left(\begin{array}{ccc} 2 & 1 & 2 \\ 1 & 2 & 1   \end{array}\right)$} $\rightarrow$
{\footnotesize$\left(\begin{array}{ccc} 1 & 2 & 3 \\ 3 & 2 & 1   \end{array}\right)$} $\rightarrow$
{\footnotesize$\left(\begin{array}{ccc} 1 & 2 & 3 \\ 2 & 1 & 2   \end{array}\right)$} $\rightarrow$
{\footnotesize$\left(\begin{array}{ccc} 1 & 2 & 3 \\ 1 & 2 & 1   \end{array}\right)$} $\rightarrow$
{\footnotesize$\left(\begin{array}{ccc} 1 & 2 & 1 \\ 3 & 2 & 1   \end{array}\right)$} $\rightarrow$
{\footnotesize$\left(\begin{array}{ccc} 2 & 1 & 2 \\ 3 & 2 & 1   \end{array}\right)$} $\rightarrow$
{\footnotesize$\left(\begin{array}{ccc} 1 & 2 & 1 \\ 2 & 1 & 2   \end{array}\right)$} $\rightarrow$
{\footnotesize$\left(\begin{array}{ccc} 3 & 2 & 1 \\ 1 & 2 & 1   \end{array}\right)$} $\rightarrow$
{\footnotesize$\left(\begin{array}{ccc} 3 & 1 & 2 \\ 2 & 1 & 3   \end{array}\right)$} $\rightarrow$
{\footnotesize$\left(\begin{array}{ccc} 1 & 2 & 1 \\ 1 & 2 & 3   \end{array}\right)$} $\rightarrow$
{\footnotesize$\left(\begin{array}{ccc} 2 & 1 & 2 \\ 1 & 2 & 3   \end{array}\right)$}.

Using the approach in the ``maximum compression example'', the path above can be transformed into a u-word of length 22 for 3-dimensional 3-permutations as follows:
\begin{equation*}
{\footnotesize
\left(\begin{array}{ccc}
1& 2 & 3\\
1 & 2 & 3 
\end{array}\right)
\rightarrow
\left(\begin{array}{cccccc}
1 & 2 & 3 & 2 & 1 & 2 \\
1 & 2 & 4 & 3 & 1 & 2
\end{array}\right)
\rightarrow
\left(\begin{array}{cccccccc}
1 & 3 & 4 & 3 & 1 & 3 & 2 & 4 \\
2 & 3 & 5 & 4 & 2 & 3 & 1 & 4 
\end{array}\right)
\rightarrow
}\end{equation*}
\begin{equation*}
{\footnotesize
\left(\begin{array}{ccccccccc}
1 & 3 & 4 & 3 & 1 & 3 & 2 & 4 & 1\\
3 & 4 & 6 & 5 & 3 & 4 & 2 & 5 & 1
\end{array}\right)
\rightarrow
\left(\begin{array}{cccccccccc}
2 & 4 & 5 & 4 & 2 & 4 & 3 & 5 & 2 & 1\\
3 & 4 & 6 & 5 & 3 & 4 & 2 & 5 & 1 & 5
\end{array}\right)
\rightarrow
}\end{equation*}
\begin{equation*}
{\footnotesize
\left(\begin{array}{ccccccccccccccc}
3 & 5 & 6 & 5 & 3 & 5 & 4 & 6 & 3 & 2 & 1 & 2 & 3 & 4 & 5 \\
3 & 4 & 6 & 5 & 3 & 4 & 2 & 5 & 1 & 5 & 6 & 5 & 1 & 5 & 1 
\end{array}\right)
\rightarrow
}\end{equation*}
\begin{equation*}
{\footnotesize
\left(\begin{array}{cccccccccccccccc}
3 & 5 & 6 & 5 & 3 & 5 & 4 & 6 & 3 & 2 & 1 & 2 & 3 & 4 & 5 & 4\\
4 & 5 & 7 & 6 & 4 & 5 & 3 & 6 & 2 & 6 & 7 & 6 & 2 & 6 & 2 & 1
\end{array}\right)
\rightarrow
}
\end{equation*}  
\begin{equation*}
{\footnotesize
\left(\begin{array}{cccccccccccccccccccccc}
3 & 5 & 6 & 5 & 3 & 5 & 4 & 6 & 3 & 2 & 1 & 2 & 3 & 4 & 5 & 4 & 5 & 4 & 1 & 2 & 1 & 2 \\
5 & 6 & 8 & 7 & 5 & 6 & 4 & 7 & 3 & 7 & 8 & 7 & 3 & 7 & 3 & 2 & 1 & 2 & 1 & 3 & 4 & 5
\end{array}\right).
}
\end{equation*}  

\section{Concluding remarks}\label{last-sec}

\begin{deff} Allowing in the definition of a u-word  consecutive columns to be considered cyclically, we define a {\em universal cycle}, or {\em u-cycle}, for $d$-dimensional $n$-permutations.
\end{deff}

For example, {\footnotesize$\left(\begin{array}{cccc} 4 & 3 & 1 & 2 \\ 4 & 1 & 3 & 2  \end{array}\right)$} is a u-cycle for 3-dimensional 2-permutations. For another example, we have the following u-cycle for 3-dimensional 3-permutations of length 16:
$$ {\footnotesize\left(\begin{array}{cccccccccccccccc} 3 & 4 & 5 & 4 & 3 & 2 & 1 & 2 & 3 & 4 & 3 & 4 & 5 & 4 & 5 & 4    \\  
 3 & 4 & 5 & 4 & 3 & 4 & 5 & 4 & 3 & 4 & 5 & 6 & 5 & 4 & 3 & 4  \end{array}\right)}.$$

Kirsch et al.~\cite{KLSS} proved the conjecture in  \cite{KPV} demonstrating that it is possible to use incomparable elements to shorten u-cycles for permutations to length $n!-i(n-1)$ for any $1\leq i\leq (n-2)!$, which extended the result in  \cite{KPV} on the existence of u-words for permutations of length  $n!+(1-i)(n-1)$. The following question is about an extension of Theorem~\ref{main-thm} and a generalization of the main result in \cite{KLSS} related to the case of $d=2$.

\begin{que} Is it true that u-cycles of length $(n!)^{d-1}-i(n-1)$, for $d\geq 2$ and $$0\leq i\leq \frac{2^{d-1}}{n-1}\left[(1+(n-1)!)^{d-1}-\left(1+\frac{(n-1)!}{2}\right)^{d-1}\right],$$ for $d$-dimensional permutations of length $n$ exist? \end{que}

For another research direction note that the procedure of removing cycles described by us in the proof of Theorem~\ref{main-thm}, gives us an immediate, but very far from the true value, lower bound of 
$$2^{\frac{2^{d-1}}{n-1}\left[(1+(n-1)!)^{d-1}-\left(1+\frac{(n-1)!}{2}\right)^{d-1}\right]}$$
for the number of different (shortened) u-words. Indeed, once we have a possibility to remove cycles, one by one, we can pick a subset of cycles to be removed that will result in a distinct u-word. Improving (significantly) the lower bound will involve a combination of our approach with usage of the {\em BEST Theorem}. 

The name ``BEST'' in the ``BEST Theorem'' is an acronym of the names of people who discovered it: N.\ G.\ de Bruijn, Tatyana Ehrenfest, Cedric Smith and W.\ T.\ Tutte.
The BEST Theorem \cite[Theorem 5b]{vA-deB1951}  (also see \cite[pp.\ 56, 68]{Stan1999} and \cite{Fred1982}) states that the number of Eulerian cycles in (traversals of) an Eulerian digraph with the vertex set $V$ and initial edge $e=v\rightarrow u$ is given by the formula
\begin{equation}\label{BEST-formula}
(\mbox{\# spanning trees rooted at }v)\cdot\prod_{x\in V}(\mbox{outdegree}(x)-1)!.
\end{equation}
Here we assume the spanning trees to be directed towards $v$, and it is known that the choice of $v$ is not important (the result will always be the same). Unfortunately,  application of the BEST Theorem in counting Eulerian paths after removing some of the cycles, and hence, counting distinct u-words, is difficult already for the cases of small $n$ and $d$, and the general case may not be feasible. 

In fact, our suggested approach in this paper may not be giving all possible u-words, as   allowing consecutive equal elements be placed not only at distance $n-1$ from each other, but also closer (which is not the case in our paper), we may enrich significantly the class of u-words/u-cycles for multi-dimensional permutations and achieve shorter lengths. For example, in the 2-dimensional case for $n=4$, in \cite{KPV}, a u-word of length 17 (< 21, the shortest length in Theorem~\ref{main-thm}) is obtained: 34321432345234343. 

\begin{que} Describe possible lengths of u-words for $d$-dimensional $n$-permutations different from those in Theorem~\ref{main-thm}. \end{que}

\begin{que} Provide a ``good'' upper bound for the number of u-words for multi-dimensional permutations. \end{que}

\vskip 3mm
\noindent {\bf Acknowledgments.} The authors are grateful to the anonymous referee for the careful reading of our paper and for providing many useful suggestions that helped improve its presentation. The first  author is supported by Leverhulme Research Fellowship (grant reference RF-2023-065\textbackslash 9).  The second author is supported in part by the Fundamental Research Funds for the Central Universities, the National Natural Science Foundation of China (12271023 and 12171034), and the Natural Science Foundation of Tianjin (24JCZDJC01390).

\end{document}